\documentclass{article}

\usepackage[utf8]{inputenc}

\usepackage{parskip, amsfonts, amsmath, changepage, amsthm, setspace, geometry, verbatim, MnSymbol,chemarrow}

\usepackage[pdftex]{graphicx}
\graphicspath{ {./images/} }

\usepackage[T1]{fontenc}

\usepackage{authblk}

\usepackage{MnSymbol, wasysym, mathtools}

\usepackage[tiny]{titlesec}

\usepackage{hyperref}

\newcommand*\tp[1]{\big( \begin{smallmatrix}#1\end{smallmatrix} \big)}

\newcommand*\ab[1]{\left\langle #1 \right\rangle}

\newcommand*\Mod[1]{ \; (\textup{mod} \; #1 )}

\newcommand*\tops[2]{\texorpdfstring{#1}{#2}}

\newcommand*\ol[1]{\overline{#1}}

\newcommand*{\Z}{\mathbb{Z}}
\newcommand*{\N}{\mathbb{N}}
\newcommand*{\R}{\mathbb{R}}
\newcommand*{\Q}{\mathbb{Q}}

\newcommand*\Cay{\textup{Cay}}

\renewcommand{\phi}{\varphi}

\newcommand*\Sec[1]{*{#1} \phantomsection \addcontentsline{toc}{section}{#1}}

\newtheorem{Thm}{Theorem}[section]

\newtheorem{Lem}[Thm]{Lemma}

\theoremstyle{definition}

\theoremstyle{remark}
\newtheorem{Ex}[Thm]{Example}
\newtheorem{Rmk}[Thm]{Remark}

\providecommand{\keywords}[1]
{
  \small	
  \textbf{\textit{Keywords---}} #1
}

\title{Chromatic numbers of Cayley graphs of abelian groups: A matrix method}
\author[a]{Jonathan Cervantes}
\author[b]{Mike Krebs}
\affil[a]{University of California, Riverside, Dept. of Mathematics, Skye Hall, 900 University Ave., Riverside, CA 92521, jcerv092@ucr.edu}
\affil[b]{California State University, Los Angeles, Dept. of Mathematics, 5151 State University Drive, Los Angeles, CA 91711, mkrebs@calstatela.edu}
\date{\today}

\begin{document}

\maketitle


\begin{abstract}

In this paper, we take a modest first step towards a systematic study of chromatic numbers of Cayley graphs on abelian groups.  We lose little when we consider these graphs only when they are connected and of finite degree.  As in the work of Heuberger and others, in such cases the graph can be represented by an $m\times r$ integer matrix, where we call $m$ the dimension and $r$ the rank.  Adding or subtracting rows produces a graph homomorphism to a graph with a matrix of smaller dimension, thereby giving an upper bound on the chromatic number of the original graph.  In this article we develop the foundations of this method.  As a demonstration of its utility, we provide an alternate proof of Payan's theorem, which states that a cubelike graph (i.e., a Cayley graph on the product $\mathbb{Z}_2\times\cdots\times\mathbb{Z}_2$ of the integers modulo $2$ with itself finitely many times) cannot have chromatic number $3$.  In a series of follow-up articles using the method of Heuberger matrices, we completely determine the chromatic number in cases with small dimension and rank, as well as prove a generalization of Zhu's theorem on the chromatic number of $6$-valent integer distance graphs.
\end{abstract}

\paragraph*{MSC2020 subject classification: }05C15

\section{Introduction}


Chromatic numbers of Cayley graphs on abelian groups have been studied in many particular cases, including integer distance graphs, circulant graphs, unit-distance graphs, cube-like graphs, Paley graphs, and so on.  The intention of the present paper is to begin to lay a foundation for a systematic unified approach to problems of this sort.

We recall the Cayley graph construction.  Given a (possibly infinite) group $G$ and a subset $S$ of $G$, we say $S$ is \emph{symmetric} if we have that $s^{-1}\in S$ whenever $s\in S$.    Given a group $G$ and a symmetric subset $S$ of $G$, we define the \emph{Cayley graph} of $G$ with respect to $S$, denoted $\text{Cay}(G,S)$, to be the graph with vertex set $G$, where two vertices $x$ and $y$ are adjacent if and only if $x=ys$ for some $s\in S$.  (Some authors require $S$ to be a generating set for $G$ in order to define the Cayley graph at all; for now we do not impose this restriction, but very shortly we will reverse course and assume after all that $S$ generates $G$.)  The set $S$ being symmetric makes $\text{Cay}(G,S)$ an undirected graph. An \emph{abelian Cayley graph} is a Cayley graph on an abelian group.  We also recall that the \emph{chromatic number} of a graph $X$, denoted $\chi(X)$, is the smallest number of colors needed to assign every vertex of $X$ a color so that if $v$ and $w$ are adjacent vertices, then $v$ and $w$ are assigned different colors.  (While the letter $G$ is often used to denote a graph, this can be misleading for Cayley graphs, as the ``$G$'' suggests a group.  We tend to use $G$ for groups and $X$ for graphs.  Some other authors prefer $\Gamma$ for groups and $G$ for graphs.)  From now on we shall assume the reader is familiar with basic terminology in graph theory, as in \cite{west_introduction_2000}, as well as basic concepts in group theory, as in \cite{Dummit_Foote_2004}.

So far as the authors have been able to ascertain, no set of methodical procedures for determining the chromatic number of an abelian Cayley graph currently exists.  As one contributor notes in a post labeled ``Chromatic numbers of infinite abelian Cayley graphs'' on the online bulletin board MathOverflow \cite{MathOverflow}, ``$\,\dots$I know very little about [this], and to my surprise, I wasn't able to find much in the literature!''  Tao expresses a similar sentiment in \cite{Tao}: ``There is a bit of literature on chromatic numbers of abelian Cayley graphs, though from a quick search I didn’t find anything that would make it substantially easier to compute the chromatic number of such graphs as compared to general graphs. Still it might be worth keeping the Cayley graph literature in mind.''

\vspace{.1in}

Given a group $G$ and a symmetric subset $S$ of $G$, let $G_e$ be the subgroup of $G$ generated by $S$.  Then the subgraph $X_e$ of $X=\text{Cay}(G,S)$ induced by $G_e$ is precisely the connected component of $X$ containing the identity element $e$.  It has been previously observed that $\chi(X_e)=\chi(X)$.  To prove this claim, it suffices to show that every connected component of $X$ is isomorphic to $X_e$.  One obtains this isomorphism by letting $T$ be a left transversal set for $G_e$ in $G$, then taking the mapping $tx\mapsto x$ for $t\in T, x\in G_e$.  Consequently, in our study of chromatic numbers of Cayley graphs, it suffices to consider connected graphs.  (\emph{Nota bene}: The existence of $T$ requires the axiom of choice (AC).  Sans AC, the statement $\chi(X_e)=\chi(X)$ may fail---see \cite{Shelah-Soifer} for an example of this phenomenon involving an abelian Cayley graph.)

Moreover, when $X=\text{Cay}(G,S)$ has a finite chromatic number, it follows from the de Bruijn-Erd\H{o}s theorem \cite{de-Bruijn-Erdos} that $\chi(X)$ equals the maximum, over all finite symmetric subsets $R$ of $S$, of $\chi(\text{Cay}(G,R))$.  We remark that AC again plays a role here, as the proof of the de Bruijn-Erd\H{o}s theorem relies on it.

For the reasons noted in the previous two paragraphs, in this paper we confine our attention to finite-degree connected abelian Cayley graphs.  That is, we assume that $G$ is an abelian group; that $S$ is a finite, symmetric subset of $G$; and that $S$ generates $G$.

\vspace{.1in}

We require that our graphs be undirected, and we do not allow multiple edges.  However, we do permit them to have loops.  These will occur whenever the identity element is in $S$.  (Recall that a \emph{loop} is an edge from a vertex to itself.)  The reason for allowing loops will become apparent as we go along---they arise naturally as a by-product of our methods.  A graph with a loop cannot be properly colored.

\subsection{Previous work}

The problem of finding the chromatic number of an abelian Cayley graph has been studied in detail in many particular instances.  Below, we briefly touch on a few of these.

\begin{itemize}

\item \textbf{Integer distance graphs.}  An \emph{integer distance graph} is a Cayley graph on the group $\mathbb{Z}$ of integers.  Chromatic numbers of such graphs have been widely investigated; see \cite{Liu} for a survey of this subject.  In Example \ref{example-ees} below, we see how our methods can be used to recover a simple but foundational result for integer distance graphs due to Eggleton, Erd\H{o}s, and Skilton.  In \cite{Zhu}, Zhu determines the chromatic number of all integer distance graphs of the form $\text{Cay}(\mathbb{Z},\{\pm a, \pm b, \pm c\})$.  In \cite{Cervantes-Krebs-Zhu}, we recover Zhu's results using our techniques, in the process obtaining improved upper bounds for the periods of optimal colorings of such graphs.

\item \textbf{Circulant graphs.} A \emph{circulant graph} is a Cayley graph on the group $\mathbb{Z}_n$ of integers modulo $n$.  Chromatic numbers of circulant graphs have been explored extensively, for example in \cite{Barajas-Serra, Heuberger, Ilic, Meszka, Nicoloso}.

\item \textbf{Unit-distance graphs.} The long-standing ``chromatic number of the plane'' problem (sometimes referred to as the Hadwiger-Nelson problem) asks for the minimum number of colors needed to assign every point in $\mathbb{R}^2$ a color so that points of distance $1$ from each other always receive different colors.  Equivalently, we can ask for the chromatic number $\chi(\mathbb{R}^2)$ of the graph $G=\text{Cay}(\mathbb{R}^2,D)$, where $D$ is the unit circle.  Subgraphs of $G$ are called \emph{unit-distance graphs}.  Currently, it is known that $\chi(\mathbb{R}^2)$ is at least $5$ and at most $7$.  The book \cite{Soifer} details the history of this problem.  The lower bound of $5$ was proved by de Grey in \cite{deGrey}, subsequent to the publication of \cite{Soifer}.  Example \ref{example-chi-of-plane-root-39} below illustrates how our methods can be used to compute chromatic numbers of unit-distance graphs in some cases.

\item \textbf{Cube-like graphs.}  Let $A$ be a finite set, and let $\mathcal{P}(A)$ be the power set of $A$, that is, the set of all subsets of $A$.  Then $\mathcal{P}(A)$ is an abelian group under the symmetric difference operation.  A \emph{cube-like graph} is a Cayley graph on $\mathcal{P}(A)$.  In \cite{Payan}, Payan proves that a cube-like graph cannot have chromatic number $3$.  Other results on chromatic numbers of cube-like graphs can be found, for example, in \cite{Kokkala-et-al} as well as \cite[Section 9.7]{Jensen}.  In Section \ref{section-payans-theorem}, we provide an alternate proof of Payan's theorem using our methods.

\item \textbf{Paley graphs.} Let $F$ be a finite field such that $q=|F|$ is congruent to $1$ modulo $4$.  Then, regarding $F$ as a group under addition, the set $S$ of quadratic residues in $F$ is symmetric.  The \emph{Paley graph} $G(q)$ is defined by $G(q)=\text{Cay}(F,S)$.  The Mathematical Reviews summary for \cite{Broere-et-al} states that ``The authors prove in this paper that the clique number and the chromatic number of the Paley graph $G(p^{2r})$ are both $p^r$, where $p$ is an odd prime and $r$ is a natural number.''

\end{itemize}

Some results of a general nature are known, and we touch on several of these below.  None of them, however, yield exact values of chromatic numbers.

\begin{itemize}

\item \textbf{Probabilistic results.} In \cite{Alon}, Alon studies chromatic numbers of random Cayley graphs, with some theorems pertaining to the particular case of random Cayley graphs on abelian groups.

\item \textbf{Spectral bounds.} Classically, Hoffman \cite{Hoffman} has a lower bound on the chromatic number of a finite graph in terms of the eigenvalues of its adjacency matrix.  For a finite Cayley graph on an abelian group $G$, these eigenvalues can in turn be expressed in terms of the irreducible characters of $G$, as discussed many places, including \cite{Krebs-Shaheen}.  This method is employed, for instance, by Vinh in \cite{Vinh} to obtain lower bounds on chromatic numbers of a certain class of finite abelian Cayley graphs.

\item \textbf{Computational complexity.} In \cite{Codenotti-et-al}, it is shown that finding the chromatic number is NP-hard for circulant graphs.  In \cite{Godsil} it is shown that, under the assumption that $P\neq NP$, no polynomial time algorithm exists for determining the chromatic number of cube-like graphs.

\item \textbf{No-hole 2-distant colorings.} A {\it no-hole 2-distant coloring} of a graph $X$ is a mapping $c$ from the vertex set $V$ of $X$ to the non-negative integers such that (i) if $u$ and $v$ are adjacent vertices, then $|c(u)-c(v)|\geq 2$, and (ii) the image of $V$ under $c$ consists of consecutive integers.  In \cite{Chang}, Chang, Lu, and Zhou prove that for a finite abelian group $G$, the Cayley graph $\text{Cay}(G,S)$ has a no-hole 2-distant coloring if and only if $G\setminus S$ generates $G$.  (Here $G\setminus S$ denotes the complement of $S$ in $G$.)\end{itemize}

\subsection{Summary of this paper and its follow-on articles}

This paper can be regarded as the first in a suite of three articles, the other two being \cite{Cervantes-Krebs-small-cases} and \cite{Cervantes-Krebs-Zhu}.  This set of papers is organized as follows.

$\bullet$ In Section \ref{section-preliminaries} of the present paper, we lay the groundwork.  We begin by showing that any finite-degree connected abelian Cayley graph can be represented by an associated integer matrix we call the ``Heuberger matrix.''  We show how basic chromatic properties of the graph, such as bipartiteness, can be found immediately from the matrix.  We then catalog several graph isomorphisms and homomorphisms induced by various row and column operations, and we give examples to show how these can be used to compute chromatic numbers in some cases.

$\bullet$ In Section \ref{section-payans-theorem} of the present paper, we use this method of Heuberger matrices to provide an alternate proof of Payan's theorem on cube-like graphs.

$\bullet$ In \cite{Cervantes-Krebs-small-cases}, we state and prove two main results which give precise and easily checked numerical conditions that completely determine the chromatic number when the associated Heuberger matrix is $2\times 2$ or $3\times 2$.

$\bullet$ In \cite{Zhu}, Zhu finds the chromatic number for an arbitrary integer distance graph of the form $\text{Cay}(\mathbb{Z},\{\pm a,\pm b,\pm c\})$, where $a, b,$ and $c$ are distinct positive integers.  Such graphs, we show, have associated $3\times 2$ matrices.  Hence, in \cite{Cervantes-Krebs-Zhu}, we demonstrate how the results of \cite{Cervantes-Krebs-small-cases} yield Zhu's theorem as a corollary.

\section{Preliminaries}\label{section-preliminaries}

In this section we establish the basic definitions and lemmas that will be used throughout the remainder of this paper as well as \cite{Cervantes-Krebs-small-cases} and \cite{Cervantes-Krebs-Zhu}.


\subsection{Standardized abelian Cayley graphs}\label{subsection-Standardized-abelian-Cayley-graphs}

For reasons noted in the introduction, we restrict ourselves to finite-degree connected abelian Cayley graphs.  Let $G$ be an abelian group written using additive notation, and let $S=\{\pm g_1,\dots,\pm g_m\}$ be a symmetric subset of $G$ that generates $G$.  Let $\mathbb{Z}$ denote the group of integers under addition, and let $\mathbb{Z}^m$ denote the $m$-fold product of $\mathbb{Z}$ with itself.
Let $e_k$ be the element of $\mathbb{Z}^m$ which equals $0$ everywhere except in the $k$th component, where it equals $1$.  Define a group homomorphism $\phi\colon\mathbb{Z}^m\to G$ by $e_k\mapsto g_k$.  Let $H$ be the kernel of $\phi$.  It is straightforward to see that $\phi$ induces a graph isomorphism between $X=\text{Cay}(\mathbb{Z}^m/H,\{H\pm e_1,\dots,H\pm e_m\})$ and $\text{Cay}(G,S)$.  In this way we standardize our graphs so that the generating set always (essentially) consists of the canonical basis vectors.

We can standardize further still.  Elements of $\mathbb{Z}^m$ are $m$-tuples of integers; we write them as column vectors.  Suppose we have $y_1,\dots,y_r\in \mathbb{Z}^m$ such that $H=\langle y_1,\dots,y_r \rangle$, the subgroup of $\mathbb{Z}^m$ generated by $y_1,\dots,y_r$.  We adopt the following convention.  To represent the graph $X$, we write the $m\times r$ matrix $M_X$ whose $j$th column is $y_j$, with the superscript $\text{SACG}$ and subscript $X$.  Examples \ref{example-cycle}--\ref{example-arbitrary-circulant-graph} illustrate the usage of this notation.  Hence, given a finite-degree connected abelian Cayley graph---even if it is of infinite order---all relevant information about it is contained in this finite matrix with integer entries.  We refer to a graph $X$ in this form as a \emph{standardized abelian Cayley graph} (hence the ``SACG''), and we call $M_X$ a \emph{Heuberger matrix} of $X$.  The authors got the idea for representing an abelian Cayley graph this way from \cite{Heuberger}, ergo the eponym.  We call $m$ the \emph{dimension} of $M_X$, and we call $r$ the $\emph{rank}$ of $M_X$.  When the graph does not need to be named, we sometimes omit the subscript.  Note that $M_X$ is not unique; a standardized abelian Cayley graph can have many different Heuberger matrices associated to it.

Often we will wish to reverse the process, that is, to define an abelian Cayley graph from a given $m\times r$ integer matrix $M$.  In this case, we take $H$ to be the subgroup of $\mathbb{Z}^m$ generated by the columns of $M$; we let $G=\mathbb{Z}^m/H$, and we take $S$ to be $\{H\pm e_1,\dots,H\pm e_m\}$.  Then $M$ is a Heuberger matrix of $\text{Cay}(G,S)$.  Hence in the sequel when we write $M_X^\text{SACG}$ without reference to a given abelian Cayley graph, we mean that $X$ equals $\text{Cay}(G,S)$ in this manner.

\begin{Ex}\label{example-cycle} Let $n\in\mathbb{Z}$.  We have that $\text{Cay}(\mathbb{Z}_n,\{\pm 1\})\cong (n)^{\text{SACG}}_X$.   Hence \[(n)^{\text{SACG}}_X\text{ is }\begin{cases}
\text{ a doubly infinite path graph } & \text{if }n=0\\

\text{ a single vertex with a loop } & \text{if }|n|=1\\

\text{ a path of length }1 & \text{if }|n|=2\\

\text{ an }|n|\text{-cycle} & \text{if }|n|\geq 3.\end{cases}\]\hfill $\square$ \end{Ex}

\vspace{.1in}

\begin{Ex}\label{example-circulant}We have that the circulant graph $\text{Cay}(\mathbb{Z}_{35},\{\pm 6,\pm 10\})$ is isomorphic to \[\begin{pmatrix}5 & 0\\
4 & 7\end{pmatrix}^{\text{SACG}}_X,\] as shown in \cite[Example 12]{Heuberger}.  Later, we will refer to this graph and others like it as ``Heuberger\linebreak circulants.''\hfill $\square$\end{Ex}

\vspace{.1in}

\begin{Ex}\label{example-MX}Consider $W=\text{Cay}(\mathbb{Z},\{\pm6,\pm10,\pm25\})$.  The graph $W$ is an example of an integer distance graph; later, we will refer to $W$ as a ``Zhu \{6,10,25\} graph.''  Taking $\phi\colon e_1\mapsto 6, e_2\mapsto 10, e_3\mapsto 25$ and all other notation as above, after some computation we find that $H$ is generated by $(5,-12,6)^t$ and $(0,5,-2)^t$.  Hence we write the graph $X$ as \(\begin{pmatrix}
5 & 0\\
-12 & 5\\
6 & -2
\end{pmatrix}^{\text{SACG}}_X,\)  with Heuberger matrix \(\begin{pmatrix}5 & 0\\
-12 & 5\\
6 & -2\end{pmatrix}.\)\hfill $\square$
\end{Ex}

\vspace{.1in}

\begin{Ex}\label{example-arbitrary-distance-graph} We now discuss matrices associated to integer distance graphs in general.  Let $a_1,\dots,a_{r+1}$ be positive integers with $\gcd(a_1,\dots,a_{r+1})=1$, where $r\geq 2$.  Let $g_k=\gcd(a_1,\dots,a_k)$ for $k=2,\dots,r$.  Let $u_{ij}$ be integers such that $a_1u_{1k}+\cdots+a_ku_{kk}=a_{k+1}g_k/g_{k+1}$ for $k=2,\dots, r$.  Using elementary number theory as in \cite{Pommersheim} to find all solutions to the linear homogeneous Diophantine equation 
$a_1 x_1+\cdots+a_{r+1} x_{r+1}=0$, we find that $\text{Cay}(\mathbb{Z},\{\pm a_1,\dots,\pm a_{r+1}\})$ is isomorphic to \[\begin{pmatrix}
\frac{a_2}{g_2} & -u_{12} & -u_{13} & \cdots & -u_{1,r-1} & -u_{1r}\\
-\frac{a_1}{g_2} & -u_{22} & -u_{23} & \cdots & -u_{2,r-1} & -u_{2r}\\
0 & \frac{g_2}{g_3} & -u_{33} & \cdots & -u_{3,r-1} & -u_{3r}\\
0 & 0 & \frac{g_3}{g_4} & \cdots & -u_{4,r-1} & -u_{4r}\\
\vdots & \vdots & \vdots & \ddots & \vdots & \vdots\\
0 & 0 & 0 & \cdots &  \frac{g_{r-1}}{g_r} & -u_{rr}\\
0 & 0 & 0 & \cdots & 0 & g_r
\end{pmatrix}^{\text{SACG}}_X.\]

Conversely, suppose $M_X$ is an $(r+1)\times r$ matrix associated to a standardized abelian Cayley graph $X$.  Let $v=(v_1,\dots,v_{r+1})^t$ be the column vector whose $j$th component equals $(-1)^j$ times the determinant of the matrix obtained by deleting the $j$th row from $M_X$.  In other words, $v$ is the generalized cross product (essentially, the exterior product) of the columns of $M_X$.  We claim that if $v\neq 0$ and $\gcd(v_1,\dots,v_{r+1})=1$, then $X$ is isomorphic to $\text{Cay}(\mathbb{Z},\{\pm v_1,\dots,\pm v_{r+1}\})$.

To prove this, take the homomorphism $\phi\colon\mathbb{Z}^{r+1}\to \mathbb{Z}$ with $\phi\colon e_j\mapsto v_j$.  Let $y_1,\dots,y_r$ be the columns of $M_X$, and let $H$ be the subgroup of $\mathbb{Z}^{r+1}$ generated by $\{y_1,\dots,y_r\}$.  We must show that $\text{ker}(\phi)=H$.  The fact that $H\subset \text{ker}(\phi)$ follows from $M_X^tv=0$, which in turn follows from cofactor expansion.  Conversely, suppose that $(c_1,\dots,c_{r+1})=c\in \text{ker}(\phi)$, or equivalently, that the dot (a.k.a. inner) product $cv=0$.  We will show that $c\in H$.  It suffices to prove this when $\text{gcd}(c_1,\dots,c_{r+1})=1$.  Let $H_\mathbb{Q}$ be the span of $\{y_1,\dots,y_r\}$ over the rational numbers $\mathbb{Q}$.  Since $v\neq 0$, it follows that $H_\mathbb{Q}$ has dimension $r$ over $\mathbb{Q}$.  Thus the orthogonal complement $H_\mathbb{Q}^\perp$ of $H_\mathbb{Q}$ in $\mathbb{Q}^{r+1}$ equals $\text{span}_\mathbb{Q} v= H_\mathbb{Q}^\perp$.  So $H_\mathbb{Q}=v^\perp$.  Hence $c\in H_\mathbb{Q}$.  In other words, $c=M_X q$ for some $q\in\mathbb{Q}^r$.  From $\text{gcd}(v_1,\dots,v_{r+1})=1$, we have from the theory of the Smith normal form \cite{Dummit_Foote_2004} that there exist unimodular matrices $U$, $V$, such that \[UM_X V=\begin{pmatrix}
I_r\\
0
\end{pmatrix},\] where $I_r$ denotes the $r\times r$ identity matrix and $0$ denotes a zero row.  

Take \[q=\left(\frac{a_1}{b_1},\dots,\frac{a_{r+1}}{b_{r+1}}\right)\]for some integers $a_i, b_i$ with $\text{gcd}(a_i,b_i)=1$ for all $i$.  Let $b=b_1\cdots b_{r+1}$.  Observe that $(UM_x V)(V^{-1} bq) = Ubc,$ so the nonzero entries of $Ubc$ are the same as those of $V^{-1} bq$.   For a tuple $w$ of integers, let $\text{gcd}(w)$ denote the greatest common denominator of its entries.  Note that if $w$ is a $k$-tuple of integers and $A$ is unimodular, then $\text{gcd}(w)=\text{gcd}(Aw)$.  (This follows by writing $A$ as a product of elementary matrices, none of which have any effect on gcd.)  Putting the pieces together, we have \[\text{gcd}(bq)=\text{gcd}(V^{-1}bq)=\text{gcd}(Ubc)=\text{gcd}(bc)=b.\]

Therefore $b|\frac{b}{b_j}\cdot a_j$ for all $j=1,\dots,r+1$.  Hence for all $j$ we have that $b_j|a_j$, which implies that $b_j=\pm 1$, because $\text{gcd}(b_j,a_j)=1$.  Therefore $q\in\mathbb{Z}^r$, showing that $c\in H$.  \hfill$\square$
\end{Ex}

\begin{Ex}\label{example-arbitrary-circulant-graph} In a similar vein, we can construct Heuberger matrices associated to arbitrary circulant graphs.  Take all notation as in Example \ref{example-arbitrary-distance-graph}, and let $n=a_{r+1}$.  Then the circulant graph $C_n(a_1,\dots,a_r):=\text{Cay}(\mathbb{Z}_n,\{\pm a_1,\dots,\pm a_r\})$ is isomorphic to \[\begin{pmatrix}
\frac{a_2}{g_2} & -u_{12} & -u_{13} & \cdots & -u_{1,r-1} & -u_{1r}\\
-\frac{a_1}{g_2} & -u_{22} & -u_{23} & \cdots & -u_{2,r-1} & -u_{2r}\\
0 & \frac{g_2}{g_3} & -u_{33} & \cdots & -u_{3,r-1} & -u_{3r}\\
0 & 0 & \frac{g_3}{g_4} & \cdots & -u_{4,r-1} & -u_{4r}\\
\vdots & \vdots & \vdots & \ddots & \vdots & \vdots\\
0 & 0 & 0 & \cdots &  \frac{g_{r-1}}{g_r} & -u_{rr}\end{pmatrix}^{\text{SACG}}_Y.\]This is the same matrix as in Example \ref{example-arbitrary-distance-graph}, but with the last row deleted.

Conversely, suppose $M_X$ is an $r\times r$ matrix.  Let $M'_X$ be the matrix obtained by deleting the first column from $M_X$.  Let $v=(v_1,\dots,v_{r})^t$ be the column vector whose $j$th component equals $(-1)^j$ times the determinant of the matrix obtained by deleting the $j$th row from $M'_X$.  Let $n=\det M_X$.  We claim that if $n\neq 0$ and $v\neq 0$ and $\gcd(v_1,\dots,v_{r})=1$, then $X$ is isomorphic to $\text{Cay}(\mathbb{Z}_n,\{\pm v_1,\dots,\pm v_{r}\})$.  To prove this, take the homomorphism $\phi\colon\mathbb{Z}^{r}\to \mathbb{Z}_n$ with $\phi\colon e_j\mapsto v_j$.  Let $y_1,\dots,y_r$ be the columns of $M_X$, and let $H$ be the subgroup of $\mathbb{Z}^{r}$ generated by $\{y_1,\dots,y_r\}$.  We must show that $\text{ker}(\phi)=H$.  The fact that $H\subset \text{ker}(\phi)$ follows from cofactor expansion.  Conversely, suppose that $(c_1,\dots,c_{r})=c\in \text{ker}(\phi)$, or equivalently, that $c\cdot v\equiv 0\Mod{n}$, where $\cdot $ denotes the dot product.  We will show that $c\in H$.  From $c\cdot v\equiv 0\Mod{n}$ we have that $c\cdot v^t = kn$ for some integer $k$.  Because $\det M_x=n$, we have that $y_1\cdot v^t=n$.  Hence $(c-ky_1)\cdot v^t =0$.  From the results of Example \ref{example-arbitrary-distance-graph}, we have that $c-ky_1$ equals a linear combination of columns of $M_X'$ with integer coefficients.  Thus $c\in H$.

We note that if $\text{gcd}(n,\text{gcd}(v))=1$, then $\text{gcd}(v)=1$.  This follows from the fact that a matrix times its adjoint equals its determinant times the identity matrix, and $v$ is a row of the adjoint.

Of course, one could just as well delete the last column instead of the first column.  We remark that an $n\times n$ square matrix $M$ with $n\geq 2$ has the property that there exists a unimodal matrix $U$ such that the gcd of the determinants of the $(n-1)\times (n-1)$ minors of the submatrix formed by deleting the last column of $MU$ is $1$ if and only if the first $n-1$ diagonal entries in the Smith normal form of $M$ are all $1$.  In \cite{Ekedahl}, Ekedahl proves that for square integer matrices, the asymptotic probability of all but one diagonal entry in the Smith normal form being 1 is approximately $0.846936$.  The precise figure involves values of the Riemann zeta function.  In a rough sense, then, we can say that at least $5/6$ of all SACGs with square Heuberger matrices of a fixed size are isomorphic to circulant graphs.  In \cite{Cervantes-Krebs-small-cases}, we give an example of an SACG with a $2\times 2$ Heuberger matrix that is not isomorphic to a circulant graph.\hfill$\square$
\end{Ex}

For a standardized abelian Cayley graph $X$, various row and column operations can be performed to an associated Heuberger matrix $M_X$ to produce an isomorphic (indeed, sometimes identical) graph $X'.$  In the following lemma, we catalog some of these.  Recall that the \emph{$\mathbb{Z}$-span} of a set $\{y_1,\dots,y_r\}$ of elements of an abelian group is the set $$\{a_1y_1+\cdots+a_ry_r\;|\;a_1,\dots,a_r\in\mathbb{Z}\},$$ that is, the set of linear combinations of $y_1,\dots,y_r$ with integer coefficients.

\begin{Lem}\label{lemma-isomorphisms}
Let $X$ and $X'$ be standardized abelian Cayley graphs with Heuberger matrices $M_X$ and $M_{X'}$, respectively.\begin{enumerate}

\item If $M_{X'}$ is obtained by permuting the columns of $M_X$, then $X=X'$.

\item If $M_{X'}$ is obtained by multiplying a column of $M_X$ by $-1$, then $X=X'$.

\item\label{item-add-multiple of column} Suppose $y_j$ and $y_i$ are the $j$th and $i$th columns of $M_X$, respectively, with $j\neq i$.  If $M_{X'}$ is obtained by replacing the 
$j$th column of $M_X$ with $y_j+ay_i$ for some integer $a$, then $X=X'$.

\item\label{item-delete-column-Z-span} If $M_{X'}$ is obtained by deleting any column from $M_X$ which is in the $\mathbb{Z}$-span of the other columns, then $X=X'$.  (In particular, deleting a zero column does not change the graph.)

\item If $M_{X'}$ is obtained by permuting the rows of $M_X$, then $X$ is isomorphic to $X'$.

\item If $M_{X'}$ is obtained by multiplying a row of $M_X$ by $-1$, then $X$ is isomorphic to $X'$.

\end{enumerate}\end{Lem}

These statements can all be proved by standard arguments.  For example, the first four items listed (the column operations) have no effect on the subgroup $H$.  The fifth operation essentially corresponds to permuting the basis vectors $e_j$; the sixth corresponds to reflecting a coordinate, i.e., mapping $e_j\mapsto -e_j$.

We remark that any finite composition of operations 1--3 is equivalent to multiplication on the right by a unimodular matrix, and any finite composition of operations 5--6 is equivalent to multiplication on the left by a signed permutation matrix.

\subsection{Basal results on chromatic numbers of abelian Cayley graphs}


\vspace{.1in}

Recall that the \emph{Cartesian product} (a.k.a. \emph{box product}) of two graphs $X$ and $Y$ with vertex sets $V(X)$ and $V(Y)$, respectively, is defined to be the graph $X\square Y$ with vertex set $V(X)\times V(Y)$, where $(x_1,y_1)$ and $(x_2,y_2)$ are adjacent if and only if either (i) $x_1=x_2$ and $y_1$ is adjacent to $y_2$ in $Y$, or else (ii) $x_1$ is adjacent to $x_2$ in $X$ and $y_1=y_2$.  Also recall (see \cite{Sabidussi}) that\begin{equation}\label{equation-chi-of-box} \chi(X\square Y)=\text{max}(\chi(X),\chi(Y)).\end{equation}

When a Heuberger matrix has a block structure, we can find the chromatic number of the associated graph by taking the maximum of the chromatic numbers of the graphs associated to the blocks.

\begin{Lem}\label{lemma-block-structure}
Suppose the standardized abelian Cayley graphs $X$ and $Y$ have Heuberger matrices $M_X$ and $M_Y$, respectively.  Define the standardized abelian Cayley graph $Z$ by \(\begin{pmatrix}M_X & 0\\
0 & M_Y\end{pmatrix}^{\emph{SACG}}_{Z}.\)  In other words, $M_Z$ is the matrix direct sum $M_X\oplus M_Y$.  Then $\chi(Z)=\text{max}(\chi(X),\chi(Y))$.\end{Lem}\begin{proof}
It is straightforward to show that $Z$ is isomorphic to $X\square Y$.  The result follows from (\ref{equation-chi-of-box}).\end{proof}

We next show that deleting a row of zeroes from a matrix has no effect on the chromatic numbers of the associated standardized abelian Cayley graphs.

\begin{Lem}\label{lemma-delete-row-all-zeroes}
Let $y_1,\dots,y_r\in \mathbb{Z}^m$, where $m\geq 2$.  For all $j=1,\dots,r$, let $y_j=(y_{1j},\dots,y_{mj})^t$.  Suppose we have the standardized abelian Cayley graph $(y_1\;\cdots\;y_r)^{\emph{SACG}}_X$.  Suppose that for some $k$ with $1\leq k\leq m$, we have that $y_{k1}=\cdots=y_{kr}=0$.  (That is, the $k$th row of $X$'s matrix has all zeroes.)  Define the standardized abelian Cayley graph $X'$ by 

\[\begin{pmatrix}y_{11} & \cdots & y_{1r}\\
\vdots & \ddots & \vdots \\
y_{k-1,1} & \cdots & y_{k-1,r}\\
y_{k+1,1} & \cdots & y_{k+1,r}\\
\vdots & \ddots & \vdots \\
y_{m1} & \cdots & y_{mr}\end{pmatrix}^{\emph{SACG}}_{X'}.\]That is, the matrix for $X'$ is the same as that for $X$, but with the $k$th row deleted.  Then $\chi(X)=\chi(X')$.\end{Lem}\begin{proof}By Lemma \ref{lemma-isomorphisms}, we have that $X$ is isomorphic to \[\begin{pmatrix}y_{11} & \cdots & y_{1r} & 0\\
\vdots & \ddots & \vdots & \vdots \\
y_{k-1,1} & \cdots & y_{k-1,r} & 0\\
y_{k+1,1} & \cdots & y_{k+1,r} & 0\\
\vdots & \ddots & \vdots  & \vdots\\
y_{m1} & \cdots & y_{mr} & 0\\
0 & \cdots & 0 & 0\end{pmatrix}^{\text{SACG}}.\]Here we permute the rows of $X$ so as to move the row of all zeroes to the bottom, then append a column of all zeroes on the right using Lemma \ref{lemma-isomorphisms}(\ref{item-delete-column-Z-span}).  The result now follows from Lemma \ref{lemma-block-structure} and Example \ref{example-cycle}.\end{proof}


\subsection{Graph homomorphisms for standardized abelian Cayley graphs}

Given two graphs $X$ and $Y$ with vertex sets $V_X$ and $V_Y$, respectively, we recall that a \emph{graph homomorphism} is a mapping $\psi\colon V_X\to V_Y$ such that if $v$ and $w$ are adjacent vertices in $X$, then $\psi(v)$ and $\psi(w)$ are adjacent vertices in $Y$.  As is well known, a proper coloring $c$ of $Y$ can be ``pulled back'' via $\psi$ to give the proper coloring $c\circ\psi$ of $X$.  Hence we obtain the following standard lemma.

\begin{Lem}\label{lemma-pullback}
Suppose $\psi$ is a graph homomorphism from $X$ to $Y$.  Then $\chi(X)\leq \chi(Y)$.\end{Lem}

With Cayley graphs, it is well known that certain group homomorphisms are graph homomorphisms.  To be precise, let $G_1$ and $G_2$ be groups, and let $S_1$ and $S_2$ be symmetric subgroups of $G_1$ and $G_2$, respectively.  Let $\psi\colon G_1\to G_2$ be a group homomorphism such that $\psi(S_1)\subset S_2$.  It follows immediately from the various definitions that $\psi$ is a graph homomorphism from $\text{Cay}(G_1,S_1)$ to $\text{Cay}(G_2,S_2)$.

Suppose $X=\text{Cay}(\mathbb{Z}^m/H,S)$ and $Y=\text{Cay}(\mathbb{Z}^\ell/J,T)$ are standardized abelian Cayley graphs, where $S=\{H\pm e_1,\dots,H\pm e_m\})$ and $T=\{J\pm e_1,\dots,J\pm e_\ell\}$.  (Note that we somewhat ambiguously use the notation $e_j$ here both for a standard basis vector in $\mathbb{Z}^m$ as well as for one in in $\mathbb{Z}^\ell$, and this will be our usual practice going forward.  We hope that in each case the context will make clear which set $e_j$ is an element of.)  Let $M_X$ and $M_Y$ be Heuberger matrices associated to $X$ and $Y$, respectively.  To indicate that $\tau\colon \mathbb{Z}^m\to \mathbb{Z}^\ell$ is a group homomorphism with $\tau(S)\subset T$, we introduce the following notation: $$(M_X)^{\text{SACG}}_X\xrightarrow[\tau]{\ocirc}(M_Y)^{\text{SACG}}_Y.$$  It is often the case that the mapping $\tau$ can be inferred from the context, and for that reason we frequently omit it.

From the Heuberger matrices associated to standardized abelian Cayley graphs, we can easily construct an assortment of graph homomorphisms.  Any group homomorphism $\tau$ from $\mathbb{Z}^m/H$ to $\mathbb{Z}^m/J$ is uniquely determined by the images of $e_1,\dots,e_m$.  (By an abuse of notation, we write $e_j$ here when we more properly should write the coset $H+e_j$.  We shall adopt this convention from now on and hope that no confusion will result.)  Let $y_j$ be the $j$th column of $M_X$.  The mapping is well-defined if and only if \begin{equation}\label{equation-column-condition}y_{1j}\tau(e_1)+\cdots+y_{mj}\tau(e_m)\in J\text{ for all } j=1,\dots,r.\end{equation}  Moreover, the requirement that $\tau$ be a graph homomorphism implies that\begin{equation}\label{equation-generator-condition}\text{for all }i=1,\dots,m,\text{ we have that }\tau(e_i)=\pm e_k\text{ for some }k=1,\dots,\ell.\end{equation}  Conversely, one can verify that any function $\tau$ on $\{e_1,\dots,e_m\}$ satisfying (\ref{equation-column-condition}) and (\ref{equation-generator-condition}) can be extended to a well-defined graph homomorphism.  The following lemma catalogues several standard graph homomorphisms we obtain in this manner.

\begin{Lem}\label{lemma-homomorphisms}\;

\begin{enumerate}

\item Every isomorphism from Lemma \ref{lemma-isomorphisms} defines a graph homomorphism.

\item\label{lemma-homomorphisms-reduce-factor} We obtain a graph homomorphism by reducing a column by a common factor.  To be precise, for any integer $a$, any  $y_1,\dots,y_r\in \mathbb{Z}^m$, and any $j=1,\dots,r$, we have \[
\begin{pmatrix}y_{11} & \cdots & ay_{1j} & \cdots & y_{1r}\\
\vdots & \ddots & \vdots & \ddots & \vdots \\
y_{m1} & \cdots & ay_{mj} & \cdots &  y_{mr}\end{pmatrix}^{\text{SACG}}
    \xrightarrow{\ocirc}
    \begin{pmatrix}y_{11} & \cdots & y_{1j} & \cdots & y_{1r}\\
\vdots & \ddots & \vdots & \ddots & \vdots \\
y_{m1} & \cdots & y_{mj} & \cdots &  y_{mr}\end{pmatrix}^{\text{SACG}}.
\]  Here the mapping is given by $e_j\mapsto e_j$.

\item\label{lemma-homomorphisms-add-rows} We obtain a graph homomorphism when we ``collapse'' the top two rows by adding them.  We can write this in terms of the associated matrices as follows:\[
\begin{pmatrix}y_{11} & \cdots & y_{1r}\\
y_{21} & \cdots & y_{2r}\\
\vdots & \ddots & \vdots \\
y_{m1} & \cdots & y_{mr}\end{pmatrix}^{\text{SACG}}_X
    \xrightarrow{\ocirc}
    \begin{pmatrix}
    y_{11}+y_{21} & \cdots & y_{1r}+y_{2r}\\
\vdots & \ddots & \vdots \\
y_{m1} & \cdots & y_{mr}
    \end{pmatrix}^{\text{SACG}}_Y.
\]That is, the top row of $M_Y$ equals the sum of top two rows of $M_X$.  Here the mapping is given by $e_1, e_2\mapsto e_1$ and $e_j\mapsto e_{j-1}$ for all $j\geq 3$.

\item\label{lemma-homomorphisms-append-columns} We obtain a graph homomorphism by appending an arbitrary column, then mapping $e_j\mapsto e_j$ for all $j$.  In other words, for all any  $y_1,\dots,y_{r+1}\in \mathbb{Z}^m$, we have \[
\begin{pmatrix}y_{11} & \cdots & y_{1r}\\
\vdots & \ddots & \vdots \\
y_{m1} & \cdots & y_{mr}\end{pmatrix}^{\text{SACG}}
    \xrightarrow{\ocirc}
    \begin{pmatrix}y_{11} & \cdots & y_{1r} & y_{1,r+1}\\
\vdots & \ddots & \vdots & \vdots \\
y_{m1} & \cdots & y_{mr} & y_{m,r+1}\end{pmatrix}^{\text{SACG}}.
\]

\item\label{lemma-homomorphisms-append-zero-row} We obtain a graph homomorphism by appending a row of all zeroes, then mapping $e_j\mapsto e_j$ for $j=1,\dots,m$.  In other words, for all any  $y_1,\dots,y_{r}\in \mathbb{Z}^m$, we have \[
\begin{pmatrix}y_{11} & \cdots & y_{1r}\\
\vdots & \ddots & \vdots \\
y_{m1} & \cdots & y_{mr}\end{pmatrix}^{\text{SACG}}
    \xrightarrow{\ocirc}
    \begin{pmatrix}y_{11} & \cdots & y_{1r}\\
\vdots & \ddots & \vdots \\
y_{m1} & \cdots & y_{mr}\\
0 & \cdots & 0\end{pmatrix}^{\text{SACG}}.
\]

\item\label{lemma-homomorphisms-composition} Any composition of the above graph homomorphisms is again a graph homomorphism.
    
\end{enumerate}

\end{Lem}

We frequently compose ``permute rows'' and ``multiply a row by $-1$'' isomorphisms with a ``sum the top two rows'' homomorphism.  Each row of the ``target matrix'' is then a sum of rows, or their negatives, from the ``source matrix.''

\subsection{Graph homomorphisms and chromatic numbers of standardized abelian Cayley graphs}\label{subsetction-homomorphisms}

The homomorphisms of the previous subsection, together with Lemma \ref{lemma-pullback}, will help us find bounds on chromatic numbers in many cases.  As a first example of this technique, we observe that we can immediately determine from an associated Heuberger matrix whether a standardized abelian Cayley graph is bipartite.

\begin{Lem}\label{lemma-bipartite}
Let $y_1,\dots,y_r\in \mathbb{Z}^m$.  For all $j=1,\dots,r$, let $y_j=(y_{1j},\dots,y_{mj})^t$.  Consider the graph $(y_1\;\cdots\;y_r)^{\text{SACG}}_X$.  Then $\chi(X)=2$ if and only if $s_j=y_{1j}+\cdots+y_{mj}$ is even for all $j$.  In other words, $X$ is bipartite if and only if all column sums are even.
\end{Lem}\begin{proof}
First, suppose all column sums are even.  If $s_j\neq 0$ for at least one value of $j$, then we have $$(y_1\;\cdots\;y_r)^{\text{SACG}}_X\xrightarrow{\ocirc}(s_1\;\cdots\;s_r)^{\text{SACG}}\xrightarrow{\ocirc}(2\;\cdots\;2\;0\;\cdots\;0)^{\text{SACG}}=(2)^{\text{SACG}}_Y.$$  
Here the first homomorphism comes from summing all of the rows, and the second comes from permuting columns to move all zeroes to the right, then reducing each nonzero column by a factor.  All columns of $(2\;\cdots\;2\;0\;\cdots\;0)$ are in the $\mathbb{Z}$-span of the first column, whence we achieve the final equality by deleting all columns except the first column.  The graph $Y$ is a path of length $1$, hence $2$-colorable.  If $s_1=\cdots=s_r=0$, then in a similar way we have a homomorphism to $(0)^{\text{SACG}}_Y$, which is a doubly infinite path graph, hence also $2$-colorable.  By Lemma \ref{lemma-pullback}, we have that $\chi(X)\leq 2$.  Since $X$ contains at least one edge, therefore $\chi(X)=2$.

Conversely, suppose that $s_j$ is odd for some $j$.  Starting at any vertex, we obtain an odd cycle by taking $y_{1j}$ steps along $H+e_1$, then $y_{2j}$ steps along $H+e_2$, and so on, finally taking $y_{mj}$ steps along $H+e_m$.  (By ``taking a step along $q$,'' we mean moving from $v$ to $v+q$.  If $y_{kj}$ is negative, then by ``taking $y_{kj}$ steps along $H+e_k$,'' we mean taking $-y_{kj}$ steps along $H-e_k$.)
\end{proof}

A nearly identical proof shows that whenever the column sums are not relatively prime, the graph is $3$-colorable.

\begin{Lem}\label{lemma-column-sums-not-coprime}
Let $y_1,\dots,y_r\in \mathbb{Z}^m$.  For all $j=1,\dots,r$, let $y_j=(y_{1j},\dots,y_{mj})^t$, and let $s_j=y_{1j}+\cdots+y_{mj}$.  Suppose we have $(y_1\;\cdots\;y_r)^{\text{SACG}}_X$.  Suppose $s_j\neq 0$ for some $j$.  If $e=\gcd(s_1,\dots,s_r)>1$, then $\chi(X)\leq 3$.\end{Lem}\begin{proof}
We have $$(y_1\;\cdots\;y_r)^{\text{SACG}}_X\xrightarrow{\ocirc}(s_1\;\cdots\;s_r)^{\text{SACG}}\xrightarrow{\ocirc}(e\;\cdots\;e\;0\;\cdots\;0)^{\text{SACG}}=(e)^{\text{SACG}}_Y.$$  
The graph $Y$ is a cycle of length $e\geq 3$, hence $3$-colorable.  By Lemma \ref{lemma-pullback}, we have that $\chi(X)\leq 3$.
\end{proof}

We require in the preceding lemma that at least one of the column sums is not zero so as to guarantee that $e$ is defined.  If all column sums are zero, then $\chi(X)=2$ by Lemma \ref{lemma-bipartite}.


\vspace{.1in}

Let $\omega\in\mathbb{C}$.  Recall that the \emph{minimal polynomial of $\omega$ over the integers}, denoted $\text{min}_\mathbb{Z}\,\omega$, is defined as $\text{min}_\mathbb{Z}\,\omega=k\,\text{min}_\mathbb{Q}\,\omega$, where $\text{min}_\mathbb{Q}\,\omega$ is the minimal polynomial of $\omega$ over the rational numbers, and $k$ is the smallest positive integer such that $k\,\text{min}_\mathbb{Q}\,\omega$ has integer coefficients.

\begin{Ex}\label{example-chi-of-plane-root-39}
In this example, we compute the chromatic number of a certain infinite unit-distance graph in the plane.  Let $\omega=\left(\frac{5}{8},\frac{\sqrt{39}}{8}\right)\in\mathbb{R}^2.$  Observe that $\omega$ is a unit vector in $\mathbb{R}^2.$  Equivalently, identifying $\mathbb{R}^2$ with the complex plane $\mathbb{C}$ so that $\omega=\frac{5}{8}+\frac{i\sqrt{39}}{8}$, we have that $|\omega|=\left|\frac{5}{8}+\frac{i\sqrt{39}}{8}\right|=1$.  Hence $\omega^2$ and $\omega^3$ are also unit vectors.  Regarding $\mathbb{R}^2$ as a group under addition, let $G$ be the subgroup generated by $\{1, \omega, \omega^2,\omega^3\}$, and let $W=\text{Cay}(G,\{\pm 1, \pm\omega, \pm\omega^2,\pm\omega^3\}).$  So $W$ is a unit-distance graph.  One can compute that $\text{min}_\mathbb{Z}\,\omega=4x^2-5x+4$.  Using this fact, with some calculation it can be shown that $W$ is isomorphic to the standardized abelian Cayley graph $$\begin{pmatrix}
4 & 0\\
-5 & 4\\
4 & -5\\
0 & 4
\end{pmatrix}^{\text{SACG}}_X.$$  Both columns of $M_X$ sum to $3$.  Therefore, by Lemmas \ref{lemma-bipartite} and \ref{lemma-column-sums-not-coprime}, we have that $\chi(W)=\chi(X)=3$.\hfill  $\square$\end{Ex}

\begin{Ex}\label{example-minimal-polynomial}Generalizing the previous example reveals a surprising connection between chromatic numbers and minimal polynomials.  Let $\omega\in\mathbb{C}$ be an algebraic number, not necessarily of unit modulus.  Let $$p(x)=\text{min}_\mathbb{Z}\,\omega=c_d x^d+\cdots + c_1 x + c_0,$$where $d$ is the degree of $p$.  Define a graph $W$ whose vertex set is $\mathbb{C}$, where two vertices are adjacent if and only if their difference equals $\omega^n$ for some nonnegative integer $n$.  We claim that if $\chi(X)\geq 4$, then $|p(1)|=|p(-1)|=1$.

To prove this claim, first observe that $W$ is precisely $\text{Cay}(\mathbb{C},S)$, where $S=\{\omega^n\;|\;n\in\mathbb{Z}\text{ and }n\geq 0\}$.  As discussed in the introduction, it follows from the de Bruijn-Erd\H{o}s theorem that if $W$ is finitely colorable, then for sufficiently large $m$ we have that $\chi(W)=\chi(X)$, where $S_m=\{\omega^n\;|\;n\in\mathbb{Z}\text{ and }0\leq n\leq m-1\}$, $G$ is the subgroup of $\mathbb{C}$ generated by $S_m$, and $X=\text{Cay}(G,S_m)$.  A straightforward induction proof shows that $X$ is isomorphic to the standardized abelian Cayley graph $Y$ with a Heuberger matrix \[M_Y=\begin{pmatrix}
c_0 & c_1 & \cdots & c_d & 0 & 0 & \cdots & 0 & \cdots & 0\\
0 & c_0 & \cdots & c_{d-1} & c_d & 0 & \cdots & 0 & \cdots & 0\\
\vdots & \vdots & \ddots & \vdots & \vdots & \vdots & \ddots & \vdots & \ddots & \vdots\\
0 & 0 & \cdots & 0 & 0 & 0 & \cdots & c_0 & \cdots & c_d\\
\end{pmatrix}^t.\]Each column sum of $M_Y$ equals $p(1)$.  Hence we must have $|p(1)|=1$, else $\chi(Y)\leq 3$ by Lemma \ref{lemma-column-sums-not-coprime}.  Similarly, by Lemma \ref{lemma-isomorphisms}, we can multiply every other row of $M_Y$ by $-1$, giving us a matrix for an isomorphic graph.  Each column sum of this new graph equals $p(-1)$, and the result follows as before.

An identical argument shows that if $|p(1)|\neq 1$ or $|p(-1)|\neq 1$, then $\chi(X)\leq 3$ for all $m$.  

We do not know whether this result is vacuous.  Hence we pose the following question: Does there exist an algebraic number $\omega\in\mathbb{C}$ such that the Cayley graph of $\mathbb{C}$ generated by the non-negative powers of $\omega$ is not $3$-colorable?\hfill$\square$\end{Ex}

\vspace{.1in}

The number of columns in a Heuberger matrix for one of our graphs can be taken to be the rank $r$ of the subgroup $H$, that is, the cardinality of a minimal generating set for $H$.  The difficulty of determining the chromatic number seems to increase as $r$ increases.  In this paper we consider cases where $r$ is small.  The cases $r=0$ and $r=1$ are fairly straightforward; the case $r=2$ is dealt with in \cite{Cervantes-Krebs-small-cases} for $m=2$ and $m=3$; and the cases $r=2, m\geq 4$ as well as $r\geq 3$ are left for future investigation.  (Here $m$ is the dimension, i.e., the number of rows in the matrix.)

When $r=0$, we have that $H$ is the trivial subgroup, so $M_X$ is a zero matrix.  In this case we have $\chi(X)=2$, by Lemma \ref{lemma-bipartite}.

The case of $r=1$, i.e., when $M_X$ has just one column, is handled by the so-called ``Tomato Cage Theorem.''   The reason for the name of this theorem is as follows.  When $r=1$ and $m=2$, we can visualize the corresponding graph by taking the infinite grid graph with vertex set $\mathbb{Z}^2$ and ``wrapping it around itself'' to form a cylindrical mesh, which reminded the authors of a tomato cage.  (It also looks a bit like a tree guard, so the term ``Tree Guard Theorem'' is applied elsewhere.)

\begin{Thm}[Tomato Cage Theorem]\label{theorem-tree-guard}
Let $y_1=(y_{11},\dots,y_{m1})^t\in\mathbb{Z}^m$, and define the standardized abelian Cayley graph $X$ by $(y_1)_X^\text{SACG}$.  If $y_1=\pm e_j$ for some $j$, then $X$ has loops and hence cannot be properly colored.  Otherwise,  \[\chi(X)=\begin{cases}

2 & \text{ if }y_1\text{ has an even number of odd entries}\\

3 & \text{ if }y_1\text{ has an odd number of odd entries.}\end{cases}\]
\end{Thm}\begin{proof}It follows immediately from the definitions that if $y_1=\pm e_j$ for some $j$, then $X$ has loops.  Suppose then that $y_1\neq\pm e_j$ for all $j$.  Let $s=|y_{11}|+\cdots+|y_{m1}|$.  Observe that $s$ is even if and only if $s_1=y_{11}+\cdots+y_{m1}$ is even if and only if $y_1$ has an even number of odd entries.  So by Lemma \ref{lemma-bipartite}, we have that $\chi(X)=2$ if and only if $y_1$ has an even number of odd entries.  If not, then $s$ is odd.  Also, because $y_1\neq\pm e_j$ for all $j$, we have that $s>1$.  We have \[
\begin{pmatrix}y_{11}\\
\vdots\\
y_{m1}\end{pmatrix}^{\text{SACG}}_X
    \xrightarrow{\ocirc}
    \begin{pmatrix}|y_{11}|\\
\vdots\\
|y_{m1}|\end{pmatrix}^{\text{SACG}}\xrightarrow{\ocirc}(s)^{\text{SACG}}_Y.
\]  The first homomorphism comes from multiplying rows by $-1$ as needed; the second comes from summing all of the rows.  Note that $Y$ is a cycle of length $s$, hence $3$-colorable.  The result follows from Lemma \ref{lemma-pullback}.\end{proof}

\begin{Ex}\label{example-ees} The following essentially appears as Theorem 10 in \cite{EES}.
Let $a,b$ be coprime positive integers and let $X$ be the integer distance graph on $\mathbb Z$ with respect to $D = \{ a, b \}$.
That is, $X = \textup{Cay}(\mathbb Z, \{ \pm a, \pm b \})$.  Let $x = (-b, a)^t$.
Then $\textup{Cay}(\mathbb Z^2 / \langle x \rangle, \{ \langle x \rangle \pm e_1, \langle x \rangle \pm e_2 \}) \cong X$, as can be seen by $e_1\mapsto a, e_2\mapsto b$.
Applying the Tomato Cage Theorem, we see that $\chi(G) = 2$ if and only if $a$ and $b$ have the same parity, and that $\chi(G) = 3$ otherwise.
\hfill $\square$
\end{Ex}

\section{Payan's theorem}\label{section-payans-theorem}

In this section we prove the following theorem.

\begin{Thm}[\cite{Payan}]\label{theorem-Payan}A cube-like graph cannot have chromatic number 3.\end{Thm}

Throughout this section we take cube-like graph to be a Cayley graph on $\mathbb{Z}_2^n$.

A special case of Theorem \ref{theorem-Payan} had previously been proven by Sokolov\'{a} in \cite{Sokolova}.  We will derive Payan's theorem from Sokolov\'{a}'s theorem, and for that reason we begin by discussing the latter.

For a positive integer $n$, the \emph{$n$-dimensional cube-with-diagonals graph} $Q^d_n$ is defined by \[Q_n^d=\text{Cay}(\mathbb{Z}_2^n,\{e_1,\dots,e_n,w_n\}),\] where $e_j$ is the $n$-tuple in $\mathbb{Z}_2^n$ with $1$ in the $j$th entry and $0$ everywhere else, and $w_n$ is the $n$-tuple in $\mathbb{Z}_2^n$ with $1$ in every entry.  We can visualize $Q^d_n$ as a hypercube with edges (called ``diagonals,'' hence the name and the superscript `$d$') added to join each pair of antipodal vertices.  Sokolov\'{a} proved that for $n$ even, $Q^d_n$ has chromatic number $4$.  We present here a condensed version of the proof in \cite{Sokolova} of this result.

\begin{Thm}[\cite{Sokolova}]\label{theorem-Sokolova}If $n$ is even, then $\chi(Q_n^d)=4$.\end{Thm}

\begin{proof}First observe that $(x_1,\dots,x_n)\mapsto (x_1,x_2+\cdots+x_n)$ defines a group homomorphism from $\mathbb{Z}^n_2$ to $\mathbb{Z}^2_2$ mapping $\{e_1,\dots,e_n,w_n\}$ to $\{(1,0),(0,1),(1,1)\}$.  So this defines a graph homomorphism from $Q^d_n$ to $Q_2^d\cong K_4$, the complete graph on $4$ vertices.  Hence $\chi(Q_n^d)\leq 4$.

Next we show that $Q^d_n$ is not properly $3$-colorable.  We do so by induction.  For the base case ($n=2$), we saw previously that $Q^d_2\cong K_4$, which is not properly $3$-colorable.  Now assume that $Q^d_n$ is not properly $3$-colorable, and we will show that $Q^d_{n+2}$ is not properly $3$-colorable.  Suppose to the contrary that $c\colon \mathbb{Z}^{n+2}_2\to\mathbb{Z}_3$ is a proper $3$-coloring.  For two tuples $v=(v_1,\dots,v_j)$ and $u=(u_1,\dots,u_k)$, we define $v*u=(v_1,\dots,v_j,u_1,\dots,u_k)$.  Define $c'\colon\mathbb{Z}^{n}_2\to\mathbb{Z}_3$ by $c'(v)=k$ if $\{c(v*(0,0)),c((v+w_n)*(1,0))\}$ equals either $\{k\}$ or $\{k,k+1\}$.  A straightforward case-by-case analysis shows that $c'$ is a proper $3$-coloring of $Q^d_n$, which is a contradiction.\end{proof}

\begin{Rmk}We briefly digress to remark that Sokolov\'{a}'s theorem can be restated as follows.  In any (not necessarily proper) $3$-coloring of the vertices of an even-dimensional hypercube, there must exist two antipodal vertices, both of which are assigned the same color.  Stated this way, it brings to mind various topological theorems such as the hairy ball theorem and the Borsuk–Ulam theorem.   
We wonder whether there might be some connection between Sokolov\'{a}'s combinatorial result and one or more of these facts from topology, perhaps along the lines of the connection between Sperner's lemma and the Brouwer fixed point theorem.\hfill$\square$\end{Rmk}

Using Heuberger matrices, we will now see how Sokolov\'{a}'s theorem implies Payan's theorem.  The key idea is to show that every nonbipartite cube-like graph contains a homomorphic image of an even-dimensional cube-with-diagonals graph.

\begin{proof}[Proof of Theorem \ref{theorem-Payan}]Let $X=\text{Cay}(\mathbb{Z}_2^n,S)$ be a nonbipartite cube-like graph.  Because $2x=0$ for all $x\in S$, there is a Heuberger matrix $M_X$ associated to $X$ whose last $m$ columns are $2e_1,\dots,2e_m$, where $m=|S|$.  That is, $M_X$ has the form $(A\;\vert\;2I_m)$ for some integer matrix $A$.  Here $I_m$ is the $m\times m$ identity matrix.  Using column operations as in Lemma \ref{lemma-isomorphisms}, we have that \[(A\;\vert\;2I_m)_X^{\text{SACG}}\cong (A'\;\vert\;2I_m)_X^{\text{SACG}},\] for some matrix $A'$ whose entries are all in $\{0,1\}$.  Because $X$ is nonbipartite, by Lemma \ref{lemma-bipartite}, some column $y$ of $A'$ contains an odd number $z$ of nonzero entries.  Hence by Lemma \ref{lemma-homomorphisms}, parts (\ref{lemma-homomorphisms-append-columns}) and (\ref{lemma-homomorphisms-append-zero-row}), we have homomorphisms\[(w_z^t\;\vert\;2I_z)_Y^{\text{SACG}}\xrightarrow[\tau_1]{\ocirc}(y\;2e_{i_1}\;\cdots \;2e_{i_z})^{\text{SACG}}\xrightarrow[\tau_2]{\ocirc}(A'\;\vert\;2I_m)_X^{\text{SACG}}\]where $i_1,\dots,i_z$ are the indices of the nonzero entries of $y$, and $w_z^t$ is a column vector of length $z$ with a $1$ in every entry.  For $\tau_1$, we insert zero rows as appropriate; for $\tau_2$ we append the requisite columns.  So $\chi(Y)\leq\chi(X)$ by Lemma \ref{lemma-pullback}.  If $z=1$, then $X$ has loops and is not properly colorable.  So assume $z\geq 3$.  Observe that $Y\cong Q_{z-1}^d$.  An application of Theorem \ref{theorem-Sokolova} then completes the proof.\end{proof}

\begin{Rmk}We note that cube-like graphs are sometimes referred to in the literature as \emph{binary Cayley graphs}, and that the cube-with-diagonals graph is also called a \emph{folded cube graph}.  A \emph{projective cube} is obtained from a hypercube by identifying antipodal vertices.  This terminology is used, for example, in \cite{Beaudou-Naserasr-Tardif}, where it is proved that every nonbipartite cube-like graph contains a projective cube as a subgraph.\end{Rmk}

\section*{Acknowledgments}

The authors wish to thank Jaan Parts for many helpful suggestions---in particular, for coining the term ``Heuberger matrix.''  We would also like to thank the referee for an exceptionally detailed read and for a large number of helpful and astute suggestions.

\providecommand{\MR}[1]{}
\providecommand{\bysame}{\leavevmode\hbox to3em{\hrulefill}\thinspace}
\providecommand{\MR}{\relax\ifhmode\unskip\space\fi MR }
\providecommand{\MRhref}[2]{%
  \href{http://www.ams.org/mathscinet-getitem?mr=#1}{#2}
}
\providecommand{\href}[2]{#2}



\end{document}